\documentclass[12pt]{amsart}
\usepackage{times}
\usepackage[T1]{fontenc}
\usepackage{dsfont}
\usepackage{mathrsfs}
\usepackage[colorlinks]{hyperref}
\usepackage{xcolor}
\usepackage[a4paper,asymmetric]{geometry}
\usepackage{mathscinet}
\usepackage{latexsym}
\usepackage{amsthm}
\usepackage{amssymb}
\usepackage{amsfonts}
\usepackage{amsmath}
\newtheorem{theorem}{Theorem}[section]
\newtheorem{thm}[theorem]{Theorem}

\newtheorem{lem}[theorem]{Lemma}
\newtheorem{proposition}[theorem]{Proposition}

\newtheorem{corollary}[theorem]{Corollary}

\theoremstyle{definition}

\newtheorem{defn}[theorem]{Definition}

\theoremstyle{remark}

\newtheorem{rem}[theorem]{Remark}
\numberwithin{equation}{section}

 \DeclareMathAlphabet{\mathpzc}{OT1}{pzc}{m}{it}

  \newcommand{\dif}{\mathrm{d}}

 \newcommand{\E}{\mathbb{E}}            % expectation
 \newcommand{\e}{\varepsilon}
 \newcommand{\p}{\partial}

 \newcommand{\N}{\mathbb{N}}
 \newcommand{\R}{\mathbb{R}}
 
 \newcommand{\PP}{\mathbb{P}}
 \newcommand{\mcl}{\mathcal}
 
 \newcommand{\Be}{\begin{equation}}
 \newcommand{\Ee}{\end{equation}}
 \newcommand{\Bs}{\begin{split}}
 \newcommand{\Es}{\end{split}}
  \newcommand{\Bes}{\begin{equation*}}
 \newcommand{\Ees}{\end{equation*}}
 \newcommand{\BT}{\begin{thm}}
 \newcommand{\ET}{\end{thm}}
 \newcommand{\Bp}{\begin{proof}}
 \newcommand{\Ep}{\end{proof}}
 \newcommand{\BL}{\begin{lem}}
 \newcommand{\EL}{\end{lem}}
 \newcommand{\BP}{\begin{proposition}}
 \newcommand{\EP}{\end{proposition}}
 \newcommand{\BC}{\begin{corollary}}
 \newcommand{\EC}{\end{corollary}}
 \newcommand{\BR}{\begin{rem}}
 \newcommand{\ER}{\end{rem}}
 \newcommand{\BD}{\begin{defn}}
 \newcommand{\ED}{\end{defn}}
 \newcommand{\BI}{\begin{itemize}}
 \newcommand{\EI}{\end{itemize}}
 
 \newcommand{\tl}{\tilde}

 \newcommand{\re}{{\rm e}}
 \newcommand{\Om}{\mcl O}
 \newcommand{\ep}{\epsilon}

\begin{document}
\title[Stochastic Gierer-Meinhardt system]
{The dynamics of the stochastic shadow Gierer-Meinhardt System}
\author[M.Winter]{Matthias Winter}
\address{Department of Mathematics, Brunel University,
Kingston Lane,
Uxbridge,
Middlesex UB8 3PH, United Kingdom}
\email{Matthias.Winter@brunel.ac.uk}
\author[L. Xu]{Lihu Xu}
\address{Department of Mathematics,
Faculty of Science and Technology
University of Macau
Av. Padre Tom\'{a}s Pereira, Taipa
Macau, China}
\email{lihuxu@umac.mo}
\author[J.Zhai]{Jianliang Zhai}
\address{School of Mathematical Sciences,
University of Science and Technology of China,
Hefei, 230026, China}
\email{zhaijl@ustc.edu.cn}
\author[T.Zhang]{Tusheng Zhang}
\address{School of Mathematics,University of Manchester
Oxford Road,
Manchester M13 9PL,
United Kingdom}
\email{tusheng.zhang@manchester.ac.uk}

%\address{Department of Mathematics, Brunel University,
%Kingston Lane,
%Uxbridge,
%Middlesex UB8 3PH, United Kingdom}
%\email{Lihu.Xu@brunel.ac.uk}
 \thanks{LX is supported by the grant SRG2013-00064-FST}

\begin{abstract} \label{abstract}
We consider the dynamics of the stochastic shadow Gierer-Meinhardt system with one-dimensional standard Brownian motion. We establish the global existence and uniqueness of solutions. We also prove a large deviation result.


\end{abstract}

\maketitle

\section{Introduction}
In his pioneering work (\cite{t1}) in 1952, Turing explained the onset of pattern formation by an instability of an unpatterned state leading to a pattern. This approach is now commonly called {\it Turing diffusion-driven
instability}. Since then many models have been studied to explore pattern formation, one of the most widely used class of models are those of the activator-inhibitor type. Among these one of the most popular models is the
Gierer-Meinhardt system which after suitable
re-scaling can be stated as follows:
\begin{equation} \label{gm1} \left\{
\begin{array}{ll}
\p_tA=\ep^2\Delta A-A+\frac{A^p}{H^q} &\mbox{ in }\Om,\\[3mm]
\tau \p_t H=D\Delta H-H+\frac{A^\alpha}{H^\beta} & \mbox{ in }\Om,\\[3mm]
\frac{\partial A}{\partial\nu}= \frac{\partial
H}{\partial\nu}=0 & \mbox{ on }\partial\Om,
\end{array} \right. \end{equation} where $ \Om
\subset \mathbb{R}^d$ is a smooth
and bounded domain and $p,q, \alpha, \beta$ are all positive with the condition $\frac{p-1}{\alpha}<\frac{
q}{\beta+1}$. Gierer and Meinhardt originally suggested this system in 1972 to model (re)generation phenomena in {\it hydra}. Since then it has been studied by many authors, in particular to understand its role in pattern formation. We refer to \cite{ww1} for more details about the recent development.

The dynamics of (\ref{gm1}) remains far from being completely understood. Let us mention a few results in this direction. Global existence has been shown by Rothe for the three-dimensional case with the powers $p=2,\,q=1,\,\alpha=2,\,\beta=0$ (\cite{r1}),
and by Jiang for $\frac{p-1}{\alpha}<1$ (\cite{j1}).
Blow-up in (\ref{gm1}) can occur for $\frac{p-1}{\alpha}>1$ since this even happens for the corresponding kinetic system (\cite{nst}).

The behaviour of the system (\ref{gm1}) stands in marked contrast to its shadow system, which is formally obtained by taking the limit $D\to\infty$.
Taking this limit we get
\begin{equation} \label{gmshadow} \left\{
\begin{array}{ll}
\p_tA=\ep^2\Delta A-A+\frac{A^p}{\xi^q} &\mbox{ in }\Om,\\[3mm]
\tau \dot{\xi}=-\xi+\frac{\overline{A^\alpha}}{\xi^\beta}, & \\[3mm]
\frac{\partial A}{\partial\nu}=0 & \mbox{ on }\partial\Om,
\end{array} \right. \end{equation}
where  $\overline{A^\alpha}=\frac{1}{|\Om|}\int_{\Om}A^\alpha\,dx$.
 and $|\Om|$ is the measure of $\Om$. It was suggested by Keener (\cite{k1}) to study the system (\ref{gmshadow}) and the name ``shadow system'' was proposed by Nishiura (\cite{nish2}).

The dynamics for (\ref{gmshadow}) has been less well studied than for (\ref{gm1}). Global existence and finite-time blow-up have been explored by Li and Ni (\cite{ln1}).
In particular, they show that for $\frac{p-1}{\alpha}<\frac{2}{d+2}$ there is a unique global solution, whereas for $\frac{p-1}{\alpha}>\frac{2}{d}$ blow-up can occur. The range $\frac{2}{d}\geq \frac{p-1}{\alpha}\geq \frac{2}{d+2}$ remains open.
\vskip 3mm

We are interested in the dynamics for the corresponding stochastic system, in which the stochastic term can be explained as some random \emph{migrations}.
Therefore we are going to consider the shadow Gierer-Meinhardt system with random migrations
in the following form:
\begin{equation} \label{e:SGM}
\begin{cases}
& \p_t u=\Delta u-u+\frac{u^p}{\xi^q}, \\
& \dif \xi=-\xi \dif t+\frac{\overline{u^\alpha}}{\xi^\beta}  \dif t+\e \xi \dif B_t, \\
& \frac{\p u}{\p \nu}=0, \\
& u(0)=v, \\
& \xi(0)=\zeta.
\end{cases}
\end{equation}
where $u(t,x,\omega):\R^{+} \times \mcl O \times \Omega\rightarrow \R^{+}$, $\xi(t,\omega):\R^{+} \times \Omega \rightarrow \R^{+} \setminus \{0\}$
and $\e>0$ is some constant and $B_t$ is one-dimensional standard Brownian motion.

To our knowledge, the only other paper for stochastic Gierer-Meinhardt type systems is \cite{ks10}, which includes two coupled stochastic PDEs with bounded and Lipschitz nonlinearity.
\cite{ks10} only proved the \emph{local} existence of the \emph{positive} stochastic solution by Da Prato-Zabczyk's approach (\cite{dpz92}).
The nonlinearity in Eq. \eqref{e:SGM} is not bounded and far from being Lipschitz, but we shall prove the \emph{global} existence of the strong positive solution.

Eq. \eqref{e:SGM} is a stochastic system which includes one deterministic PDE and one SDE with long range interactions. To our knowledge, this seems to
be the first paper to study this type of stochastic systems. On the other hand, Eq. \eqref{e:SGM} can be taken as a highly degenerate stochastic
PDEs (see \cite{HM06} for more details), its ergodicity is a very challenging problem which will be studied in future papers (see \cite{HM06, KS-2002} for some work in this direction).
\vskip 3mm

Our main result on global existence can be stated as follows:
\begin{thm} \label{t:MaiSol}
Let $p,q,\alpha, \beta$ satisfy the following condition
$$\frac{p-1}{\alpha}<\frac{
q}{\beta+1},\ \ \ \ \frac{p-1}{\alpha}<\frac{2}{d+2}.$$
Eq. \eqref{e:SGM} has a unique global solution $(u, \xi) \in C([0,T];C(\mcl O,\R) \times \R)$ for all $T>0$ such that
for all $t>0$
$$u(t,x) \ge 0 \ \ \ \forall x \in \mcl O,  \ \ \ \ \ \ \xi(t) \ge \re^{-\frac 32 t-\e |B_t|} \zeta.$$
\end{thm}
\noindent The large deviation principle will be introduced in Section 4 and its main result will be given in Theorem \ref{T:LDP} below.
 As references for large deviation results on stochastic systems, we give the following list of articles which is far from being complete:  \cite{BD}-\cite{fr88},
\cite{lrz13}, \cite{pe94}-\cite{sz11}, \cite{xz09}-\cite{zh02}.

We shall follow the approach in \cite{ln1} to prove Theorem 1, some ideas along the same lines have also appeared in
\cite{j1,nish2}. The random force in Eq. \eqref{e:SGM} produces some additional stochastic terms, which can be very large or even become infinite.
To control these terms, we shall use a martingale inequality and
modify the energy estimate in \cite{ln1} by adding suitable stochastic terms and figuring out an explicit inequality.
For the large deviation result, we shall follow the variational approach in \cite{BD} by checking the two assumptions of
Theorem 4.4 therein (see Propositions \ref{p:LDP2} and \ref{p:LDP3} below). To prove these two propositions,
we also need to use a martingale inequality and some special energy estimates.
\vskip 3mm

The structure of this paper is as follows. In Section 2 we show local existence and uniqueness of solutions. In Section 3 we prove global existence and uniqueness. In Section 4 we prove the large deviation result.
Finally, in Section 5 we discuss our results and give an outlook to open problems and further research.

\section{Local existence and uniqueness of the shadow stochastic Gierer-Meinhardt system}
Without loss of generality, we assume that $\e=1$ in this and the next section.
Write
$$B^{*}_t=\sup_{0 \le s \le t} |B_s| \ \ \ \ \forall t>0,$$
let $N>0$ be a constant and define the following stopping time
$$\tau_N(\omega)=\inf\{t>0: |B_t(\omega)| \ge N\}.$$
It is clear that
\Be
\{\omega \in \Omega:\tau_N(\omega) \le t\}=\{\omega \in \Omega: B^{*}_t(\omega) \ge N\}.
\Ee
It is well known that $\sup_{0 \le s \le t} B_s$ satisfies
\Bes
\PP\left(\sup_{0 \le s \le t} B_s \in (x, x+\dif x)\right)=\frac{2}{\sqrt{2 \pi t}}\re^{-\frac{x^2}{2t}} \dif x, \ \ \ x>0.
\Ees
Since
\Bes
\begin{split}
\PP\left(B^*_t>x\right) & \le \PP\left(\sup_{0 \le s \le t} B_s>\frac x2\right)+\PP\left(\sup_{0 \le s \le t}(-B_s)>\frac x2\right) \\
&=2\PP\left(\sup_{0 \le s \le t} B_s>\frac x2\right)=\frac{4}{\sqrt{2 \pi t}}\int_{\frac{x}{2 \sqrt t}}^\infty\re^{-\frac{y^2}{2}} \dif x,
\end{split}
\Ees
the distribution of $B^*_t$ has a density function $f_t$ satisfying
\Be \label{e:DenB*t}
f_t(x) \le \frac{4}{\sqrt{2 \pi t}} \re^{-\frac{x^2}{8t}}.
\Ee
%\Be \label{e:OmeN}
%\Omega_N=\left\{\omega: B^{*}_1 \le N\right\},
%\Ee
%it is clear that
%$$\lim_{N \rightarrow \infty} \PP(\Omega_N)=1.$$
For notational simplicity, we shall drop the variable $\omega$ in the random variables or random sets below if no confusions arise.
Further define
\Be \label{e:SRDef}
S(t)=\re^{(\Delta-1) t}, \ \ \ R(t,B_t)=\re^{-\frac 32 t+B_t},
\Ee
where $\Delta$ is the Laplace operator with Neumman boundary condition
$C({\mcl O};\R^d)$ is the space of all bounded continuous functions $f:{\mcl O} \rightarrow \R^d$ with uniform norm. It is easy to check that $C({\mcl O},\R^d)$ is closed under uniform norm.
For notational simplicity, we shall write
$$\|f\|_{C}=\|f\|_{C({\mcl O},\R^d)} \ \ \ \ \ \forall \ \ f \in {C({\mcl O},\R^d)}.$$
It is clear that the following relations hold:
\Be \label{e:StConFpNor}
\begin{split}
& \|S(t) f\|_{{C}} \le \|f\|_{{C}} \ \ \ \ \ \forall  t>0 \ \ \forall  f \in {C({\mcl O},\R^d)}, \\
& \|f^p\|_{{C}} \le \|f\|^p_{{C}} \ \ \ \ \ \forall  p \ge 1 \ \ \forall  f \in {C({\mcl O},\R^d)}.
\end{split}
\Ee
%\Be \label{e:L2StfEst}
%\|\nabla S(t) f\|_{L^2}\le C t^{-\frac 12} \|f\|_{L^2}.
%\Ee
For any $(u,\xi)$, recall
$$\|(u,\xi)\|_{C([0,T];C \times \R)}=\|u\|_{C([0,T];C)}+\|\xi\|_{C([0,T];\R)} \ \ \ \ \ \forall \ T>0.$$
Let $X, Y$ both be some quantities, we shall simply denote $Y\lesssim X$ if there exists some (not
important constant) $C$ such that $Y \le C X$.

\begin{lem} \label{l:LocSol}
For every $N>0$, there exists some $T$ depending on $N, \|v\|_{C}$ and $\zeta$ such that
for all $\omega \in \Omega$ up to a negligible set, Eq.
\eqref{e:SGM} has a unique solution $(u,\xi) \in C([0,T \wedge \tau_N]; C({\mcl O},\R)\times \R)$ such that for all $t \in [0,T \wedge \tau_N]$
\Be \label{e:BFix}
\begin{split}
& u(t)=S(t) v+\int_0^t S(t-s) \left(\frac{u^p(s)}{\xi^q(s)}\right) \dif s,\\
& \xi(t)=R(t,B_t) \zeta+\int_0^t R(t-s,B_t-B_s) \left(\frac{\overline{u^\alpha}(s)}{\xi^\beta(s)}\right) \dif s,
\end{split}
\Ee
with the property
\Be \label{e:XiEst}
\xi(t) \ge {\rm e}^{-\frac 32 t-N} \zeta \ \ \ \ \forall t \in [0,T \wedge \tau_N].
\Ee
Moreover, $(u(t),\xi(t))$ satisfies the first two equations in Eq. \eqref{e:SGM} for each $t \in (0,T \wedge \tau_N]$.
In particular,
$$\xi(t)=\zeta-\int_0^t \xi(s) \dif s+\int_0^t \frac{\overline{u^\alpha}(s)}{\xi^\beta(s)} \dif s+\int_0^t \xi(s)\dif B_s \ \ \ \forall t \in [0,T \wedge\tau_N].$$
\end{lem}
\begin{proof}
For all $\omega \in \Omega$ up to a negligible set,
define the following space
\Bes
\begin{split}
\mcl A_{T,M,N,\omega}=\bigg\{\big(u(\omega), &\xi(\omega)\big) \in  C([0,T \wedge \tau_N(\omega)]; C({\mcl O},\R) \times \R^+):\\
& u(\omega,t) \ge 0, \ \xi(\omega,t) \ge {\rm e}^{-\frac 32 t-N} \zeta, \  \ \forall \ 0 \le t \le T \wedge \tau_N(\omega);\\
&  u(0)=v, \ \xi(0)=\zeta; \ \|(u,\xi)(\omega)\|_{C([0,T\wedge \tau_N(\omega)];C \times \R)} \le M.\bigg\},
\end{split}
\Ees
where $T \in (0,1]$ is some number depending on $M,N,v,\zeta$ to be determined later and
$$M>2+\|v\|_{C}+\re^N \zeta.$$
We shall drop all the $\omega$ in the definition of $\mcl A_{T,M,N,\omega}$ in the argument below for notational simplicity.

For all $(u_1,\xi_1), (u_2,\xi_2) \in \mcl A_{T,M,N}$, define
\Bes
\begin{split}
{\rm d}_T\left((u_1,\xi_1), (u_2, \xi_2)\right)=\|(u_1,\xi_1)-(u_2, \xi_2)\|_{C([0,T \wedge \tau_N];C\times \R)}.
\end{split}
\Ees
It is easy to check that under the distance ${\rm d}_T$ the set $\mcl A_{T,M,N}$ is a closed metric space.

For each $(u,\xi) \in \mcl A_{T,M, N}$, define
\Be \label{e:DefMapF}
\begin{split}
& \left[\mcl F_1 (u,\xi)\right](t)=S(t) v+\int_0^t S(t-s) \left(\frac{u^p(s)}{\xi^q(s)}\right) \dif s,\\
& \left[\mcl F_2 (u,\xi)\right](t)=R(t,B_t) \zeta+\int_0^t R(t-s,B_t-B_s) \left(\frac{\overline{u^\alpha}(s)}{\xi^\beta(s)}\right) \dif s,
\end{split}
\Ee
where $S$ and $R$ are defined in \eqref{e:SRDef}. For further use, we simply denote
$$\mcl F(u, \xi)=\left(\mcl F_1(u,\xi), \mcl F_2(u,\xi)\right).$$
We shall prove below that
\begin{itemize}
\item[(i)] There exists some $\hat T$ depending on $N, M, \|v\|_{C}$ and $\zeta$  such that
\Be \label{e:B1Inc}
\mcl F(u, \xi) \in \mcl A_{T,M, N}
\Ee
 for any $(u,\xi) \in \mcl A_{T,M, N}$ with $T=\hat T$.
\item[(ii)] There exists some $\tl T$ depending on $N, M, \|v\|_{C}$ and $\zeta$ such that
\Be \label{e:B2Con}
\dif_{T}(\mcl F(u_1, \xi_1),\mcl F(u_2, \xi_2)) \le \frac 12 \dif_{T}((u_1, \xi_1),(u_2, \xi_2))
\Ee
 for any
$(u_1, \xi_1), (u_2, \xi_2) \in \mcl A_{T,M, N}$ with $T=\tl T$.
\end{itemize}
By the definition of $\mcl A_{T,M,N}$, taking $T=\min\{\tl T, \hat T\}$, it is clear that \eqref{e:B1Inc} holds for any $(u,\xi) \in \mcl A_{T,M, N}$ and that
\eqref{e:B2Con} holds for any $(u_1,\xi_1), (u_2,\xi_2) \in \mcl A_{T,M, N}$. Thus, we apply Banach fixed point theorem to
obtain a local unique solution in the sense of \eqref{e:BFix}. Differentiating both sides of \eqref{e:BFix} (\cite{iw81}), we immediately get that
$(u,\xi)$ satisfies the first two equations of Eq. \eqref{e:SGM} and that the desired stochastic integral equation holds.

Now we only need to show the statements (i) and (ii) from above. Let $C$ be some positive constants depending only on $\alpha, \beta, p, q$, whose exact values may vary from case to case.
\vskip 2mm

Let us first show (i).  For any $(u,\xi) \in \mcl A_{\hat T,M, N}$ with $\hat T$ to be determined below, it is clear $\mcl F(u, \xi)(0)=(v,\zeta)$.
Since $S(t)$ maps a positive function to a positive one, it is easy to see
\Bes
\left[\mcl F_1(u,\xi)\right](t) \ge 0 \ \ \ \ \ \forall t \in [0,\hat T \wedge \tau_N]. % {\rm e}^{-\frac 32 t+B_t} \zeta \ge {\rm e}^{-\frac 32 t-N} \zeta \ge {\rm e}^{-\frac 32-N} \zeta.
\Ees
By \eqref{e:StConFpNor}, for all $t \in [0, \hat T \wedge \tau_N]$ we have
\Bes
\begin{split}
\|\left[\mcl F_1(u,\xi)\right](t)\|_{C}
& \le \|v\|_{{C}}+{\rm e}^{\frac 32 q+Nq} \zeta^{-q} \int_0^t \|u(s)\|^p_{{C}} \dif s \\
& \le \|v\|_{{C}}+{\rm e}^{\frac 32 q+Nq} \zeta^{-q} M^p t,
\end{split}
\Ees
and
\Bes
\begin{split}
|\left[\mcl F_2 (u,\xi)\right](t)| & \le {\rm e}^{-\frac 32 t+B_t} \zeta+{\rm e}^{\frac 32 \beta t+N \beta}\int_0^t {\rm e}^{-\frac 32 (t-s)+B_t-B_s} \|u(s)\|^\alpha_{C} \dif s \\
& \le {\rm e}^{N} \zeta+{\rm e}^{\frac 32 \beta+N \beta+2N} M^\alpha t
\end{split}
\Ees
Taking $\hat T=\min\{T_1,T_2\}$ with $T_1={\rm e}^{-\frac 32 q -Nq} \zeta^{q} M^{-p}$
and $T_2={\rm e}^{-\frac 32 \beta-N \beta-2N} M^{-\alpha}$, from the above two inequalities we get
\Bes
\begin{split}
\|\mcl F(u,\xi)\|_{C([0, \hat T \wedge \tau_N];C \times \R)} \le 2+\|v\|_{{C}}
+{\rm e}^{N} \zeta \le M.
\end{split}
\Ees
Hence,
$\mcl F(u,\xi) \in \mcl A_{\hat T,M,N}.$
\vskip 2mm

Next we show (ii). For any $(u_1,\xi_1),(u_2,\xi_2) \in \mcl A_{\tl T,M, N}$ with $\tl T$
to be determined below,
observe that for all $t \in [0,\tl T \wedge \tau_N]$
\Bes % \label{e:DifF1}
\begin{split}
\|\left[\mcl F_1(u_1,\xi_1)\right](t)-\left[\mcl F_1(u_2,\xi_2)\right](t)\|_{C} & \le
\int_0^t \left\|\frac{u_1^p(s)}{\xi_1^q(s)}-\frac{u_2^p(s)}{\xi_2^q(s)}\right\|_{C} \dif s \le I_1(t)+I_2(t)
\end{split}
\Ees
where
\Bes
\begin{split}
& I_1(t)=\int_0^t \frac{\left\|u_1^p(s)-u^p_2(s)\right\|_{C}}{\xi_1^q(s)} \dif s,\\
& I_2(t)=\int_0^t \|u_2^p(s)\|_{C}\left|\frac{1}{\xi_1^q(s)}-\frac{1}{\xi_2^q(s)}\right| \dif s.
\end{split}
\Ees
Writing $u_{1,2, \lambda}(s)=\lambda u_1(s)+(1-\lambda) u_2(s)$ for $\lambda \in [0,1]$, by \eqref{e:StConFpNor} we have
\Be
\begin{split}
\left\|u_1^p(s)-u^p_2(s)\right\|_{C} & \le p\int_0^1 \left\|\left(u_{1,2, \lambda}(s)\right)^{p-1} (u_1(s)-u_2(s))\right\|_{C} \dif \lambda \\
& \le p\int_0^1 \left\|u_{1,2, \lambda}(s)\right\|_{C}^{p-1} \left\|u_1(s)-u_2(s)\right\|_{C} \dif \lambda \\
& \le  p M^{p-1} \left\|u_1(s)-u_2(s)\right\|_{C}.
\end{split}
\Ee
Thus
$$I_1(t)\le p {\rm e}^{\frac 32 q+Nq} \zeta^{-q} M^{p-1} t \left\|u_1-u_2\right\|_{C([0,t];C)} \ \  \ \ \ \forall t \in [0,\tl T \wedge \tau_N].$$
Writing $\xi_{1,2, \lambda}(s)=\lambda \xi_1(s)+(1-\lambda) \xi_2(s)$ for $\lambda \in [0,1]$, we have
\Bes
\begin{split}
I_2(t) & \le q\int_0^t M^p \int_0^1 \frac{|\xi_1(s)-\xi_2(s)|}{\left(\xi_{1,2,\lambda}(s)\right)^{q+1}} \dif \lambda \dif s \\
& \le q \re^{(\frac 32+N)(q+1)}\zeta^{-(q+1)} M^p  t\|\xi_1-\xi_2\|_{C([0,t];\R)} \ \ \ \ \ \ \forall t \in [0,\tl T \wedge \tau_N],
\end{split}
\Ees
which, together with the estimate of $I_1$, implies that for all $t \in [0,\tl T \wedge \tau_N]$
\Be
\begin{split}\label{e:2.12}
& \ \ \ \ \ \|\mcl F_1(u_1,\xi_1)-\mcl F_1(u_2,\xi_2)\|_{C([0,t];C)} \\
& \le C {\rm e}^{(\frac 32+N)q} \zeta^{-q} M^{p-1}\left(1+\re^{\frac 32+N} M\zeta^{-1}\right)t \left\|(u_1,\xi_1)-(u_2,\xi_2)\right\|_{C([0,t];C \times \R)}.
\end{split}
\Ee
A similar argument as above gives that for all $ t \in [0,\tl T \wedge \tau_N]$ ,
\Be
\begin{split}\label{e:2.13}
& \ \ \ \ \ \ \ \|\mcl F_2(u_1,\xi_1)-\mcl F_2(u_2,\xi_2)\|_{C([0,t];\R)}  \\
 & \le C  {\rm e}^{2N+(\frac 32+N)\beta} \zeta^{-\beta} M^{\alpha-1}(1+\re^{\frac 32+N} M\zeta^{-1})t \left\|(u_1,\xi_1)-(u_2,\xi_2)\right\|_{C([0,t];C \times \R)}.
\end{split}
\Ee
From the above two inequalities, there exists some $\tl T$ depending on $M,N,\zeta$ such that
\Bes
\|\mcl F(u_1,\xi_1)-\mcl F(u_2,\xi_2)\|_{C([0,\tl T \wedge \tau_N];C \times \R)} \le \frac 12
\left\|(u_1,\xi_1)-(u_2,\xi_2)\right\|_{C([0,\tl T\wedge \tau_N];C \times \R)}
\Ees
i.e.,
$$\dif_{\tl T}(\mcl F(u_1,\xi_1),\mcl F(u_2,\xi_2)) \le \frac 12 \dif_{\tl T}((u_1,\xi_1),(u_2,\xi_2)).$$
\end{proof}

%\begin{corollary}
%For every $\omega \in \Omega$ up to a negligible set, there exists some $T(\omega) \in (0,1)$ such that
%$(u,\xi) \in C([0,T(\omega)]; C \times \R^+)$.
%\end{corollary}
%\begin{proof}
%Define $\tl \Omega_N=\{\omega: \tau_N \le 1\}$, it is easy to see
%$\lim_{N \rightarrow \infty} \PP(\tl \Omega_N)=1$. For every $\omega \in
%\tl \Omega_N$, from Lemma
%\ref{l:LocSol} Eq. \eqref{e:SGM} admits a unique solution $(u,\xi) \in C([0,T(\omega)]; C \times \R)$ with
% $T(\omega)=T \wedge \tau_N(\omega)$.
%\end{proof}

\section{Global existence and uniqueness of the shadow stochastic Gierer-Meinhardt system}
\subsection{Some a'priori estimates}
To prove the global existence and uniqueness theorem, we assume that $(u(t),\xi(t))_{0 \le t \le 1}$ is a solution of Eq. \eqref{e:SGM} such that
$$u \in C([0,1];C({\mcl O},\R)), \ \ \ \xi \in C([0,1],\R) \ \ \ \ a.s., $$
and prove the following a'priori estimates of $(u,\xi)$.
\begin{lem}  \label{l:XiBou}
We have
\Be \label{e:XiBou}
\begin{split}
\xi(t)\ge \re^{-\frac 32 t+B_t} \zeta \ \ \ \ \forall t>0,
\end{split}
\Ee
\Be \label{e:XiGeInf}
\begin{split}
\inf_{0 \le s \le t}\xi(s)\ge \re^{-\frac 32 t-B^*_t} \zeta  \ \ \ \ \forall t>0,
\end{split}
\Ee
\Be \label{e:XiLeSup}
\begin{split}
\sup_{0 \le t \le 1}\xi(t) \lesssim \re^{B^*_1} \zeta+\re^{2 B^*_1}\left(\sup_{0 \le t \le 1}\overline{u^{\alpha}}(t)\right)^{\frac 1{1+\beta}}.
\end{split}
\Ee
\end{lem}
\begin{proof}
Applying It$\hat{o}$ formula to $\xi^{1+\beta}(t)$ we have
\Be
\begin{split}
\dif \xi^{1+\beta}(t)=\frac 12(1+\beta)(\beta-2) \xi^{1+\beta}(t) \dif t+(1+\beta) \xi^{1+\beta}(t) \dif B_t+(1+\beta)\overline{u^{\alpha}}(t) \dif t,
\end{split}
\Ee
which implies
\Be \label{e:BiBou1}
\begin{split}
\xi^{1+\beta}(t)&=\re^{-\frac 32(1+\beta)t+(1+\beta)B_t} \zeta^{1+\beta} \\
& \ \ \ +(1+\beta)\int_0^t
\re^{ -\frac 32(1+\beta)(t-s)+(1+\beta)(B_t-B_s)}  \overline{u^{\alpha}}(s) \dif s,
\end{split}
\Ee
which clearly implies the desired three inequalities.
\end{proof}
\ \  \

%Define
%$$\Lambda(t)=\int_0^t \frac{\overline{u^{\alpha}(s)}}{\xi^{1+\beta}(s)}\dif s.$$
Let $\delta>0$ be some fixed number and define
\Bes
\mcl M_{\delta}(t)=\int_0^t \xi^{-\delta}(s) \dif B_s, \ \ \ \ \mcl M^*_\delta=\sup_{0 \le t \le 1} \mcl M_\delta(t).
\Ees

\begin{lem} \label{l:EstOme}
For all $M>0$ we have
\Be
\E \mcl M^*_\delta \le C
\Ee
where $C$ depends only on $\delta, \zeta$. Moreover, we have
\Bes
\mcl M^*_\delta<\infty \ \ \ \ a.s..
\Ees
\end{lem}
\begin{proof}
It follows from the martingale inequality and It$\hat{o}$ isometry that
\Bes
\begin{split}
\E \mcl M^*_\delta & \le \left[\E \sup_{0 \le t \le 1} \left|\int_0^t \xi^{-\delta}(s) \dif B_s\right|^2\right]^{\frac 12} \\
 &\le \sqrt 2 \left[\E\left|\int_0^1 \xi^{-\delta}(s) \dif B_s\right|^2\right]^{\frac 12}=\sqrt 2 \left[\int_0^1 \E \xi^{-2\delta}(s) \dif s\right]^{\frac 12} .
\end{split}
\Ees
This and \eqref{e:XiBou} further give
$$\E \mcl M^*_\delta \le \sqrt 2 \zeta^{-\delta} \int_0^1 \E \re^{3 \delta t-2\delta B_t} \dif s,$$
which immediately implies the desired inequality.
\end{proof}

\begin{lem} \label{l:IntBarUXi}
Let $\delta>0$. We have
\Be \label{e:IntBarUXi}
\int_0^1 \frac{\overline{u^{\alpha}}(s)}{\xi^{1+\beta+\delta}(s)}\dif s \le \Lambda(\delta,\zeta,B,\mcl M^*_\delta),
\Ee
where
\Bes
\begin{split}
\Lambda(\delta, \zeta, B, \mcl M^*_\delta)=\delta^{-1} \zeta^{-\delta}+\frac{3+\delta}{2} \re^{\frac 32 \delta+\delta B^*_1}\zeta^{-\delta}+\mcl M^*_\delta.
\end{split}
\Ees
\end{lem}
\begin{proof}
Applying It$\hat{o}$ formula to $\xi^{-\delta}(t)$, we get
\Bes
\begin{split}
\xi^{-\delta}(t)-\zeta^{-\delta}=\frac{\delta(3+\delta)}{2} \int_0^t \xi^{-\delta}(s) \dif s-\delta \int_0^t \frac{\overline{u^{\alpha}}(s)}{\xi^{1+\delta+\beta}(s)} \dif s-\delta \int_0^t \xi^{-\delta}(s) \dif B_s,
\end{split}
\Ees
which gives
\Bes
\begin{split}
\int_0^t \frac{\overline{u^{\alpha}}(s)}{\xi^{1+\delta+\beta}(s)} \dif s
& \le \delta^{-1} \zeta^{-\delta}+\frac{3+\delta}{2} \int_0^t \xi^{-\delta}(s) \dif s+\sup_{0 \le t \le 1}\left|\int_0^t \xi^{-\delta}(s) \dif B_s\right| \\
& \le \delta^{-1} \zeta^{-\delta}+\frac{3+\delta}{2} \int_0^t \re^{\frac 32 \delta s-\delta B_s} \zeta^{-\delta} \dif s+\sup_{0 \le t \le 1}\left|\int_0^t \xi^{-\delta}(s) \dif B_s\right|
\end{split}
\Ees
where the last inequality is by \eqref{e:XiBou}. This immediately yields the desired inequality.
\end{proof}

Next we shall follow the spirit in \cite{ln1} to prove the following energy estimates, which is the key point for
establishing the global solution.
\begin{lem} \label{l:ExtSol}
Let $\rho>0$ be some number such that
\Be \label{e:RhoCon}
\rho<\frac{q}{1+\beta}, \ \ \ \ \ \frac{p-1}{\alpha}<\rho<\frac{2}{d+2}.
\Ee
Let $\ell>0$ and let
$$\theta=\frac{1}{\ell}(p-1-\alpha \rho+\ell), \ \ \gamma=\frac{d(\rho+\theta-1)}{2 \theta}.$$
Let $\delta \in (0, \frac{q-\rho-\rho \beta}{\rho})$.
As $\ell$ is sufficiently large so that $\theta \in (0,1)$, $\gamma \in (0,1)$ and $\frac{\rho}{1-\gamma \theta} \in (0,1)$,
we have
\Be \label{e:PriEst}
\sup_{0 \le t \le 1} \|u(t)\|^\ell_{L^\ell} \le C \left(\|v\|_{L^\ell}^{\frac{(1-\theta \gamma)\ell}{1-\theta}}+\Theta^{\frac{1-\theta \gamma}{1-\theta}} \Lambda^{\frac{\rho}{1-\theta}}(\delta,\zeta,B,\mcl M^{*}_\delta)\right) \vee 1
\Ee
where $C$ depends on $p,q, \alpha, \beta$ and $\Lambda(\delta,\zeta,B,\mcl M^*_\delta)$ is defined in Lemma \ref{l:IntBarUXi} and
$$\Theta=\re^{\frac 32\frac{q-\rho(1+\beta+\delta)}{1-\theta \gamma}}
\zeta^{\frac{\rho(1+\beta+\delta)-q}{1-\theta \gamma}} \re^{\frac{q-\rho(1+\beta+\delta)}{1-\theta \gamma}B_1^*}.$$
\end{lem}
\begin{proof}
Without loss of generality, we assume $|\mcl O|=1$ in this proof.
Let $\ell$ be a large number to be chosen later and write
$$w(t)=u^{\ell/2}(t).$$ Then
a straightforward calculation gives
\Be \label{e:WNorEqn}
\p_t \|w\|^2_{L_2}=-\frac{4d(\ell-1)}{\ell}\|\nabla w\|_{L_2}^2-\ell \|w\|^2_{L_2}+\frac{\ell}{\xi^q} \int_{\mcl O} u^{p-1+\ell}\dif x.
\Ee \label{e:RhoEqn}
%where $\hat \beta>\beta$ so that
%Write
%$$\rho=\frac{q}{1+\hat \beta}, \ \ \ \theta=\frac{1}{\ell}(p-1-\alpha \rho+\ell,$$
%where $\hat \beta>\beta$ so that
%$$\rho \le \frac{2}{n+2},$$
Note that $\theta \in (0,1)$ as $\ell$ is large and $\lim_{\ell \rightarrow \infty} \theta=1$.
By the second inequality of \eqref{e:RhoCon} we have
\Be \label{e:GamThe<1}
0<\gamma<1 \ \ \ \ \ \ {\rm as} \ \ \ell \ \ {\rm is \ \ sufficiently \ \ large,}
\Ee
by H\"{o}lder inequality and the following Gagliardo-Nirenberg inequality
\Be
\begin{split}
\|w\|_{L^{\frac{2 \theta}{1-\rho}}} \le C \left(\|\nabla w\|_{L^2}+\|w\|_{L^2}\right)^{\gamma}\|w\|_{L^2}^{1-\gamma},
\end{split}
\Ee
we have
\Be \label{e:LasEneEst}
\begin{split}
\frac{1}{\xi^q} \int_{\mcl O} u^{p-1+\ell}\dif x&=\frac{1}{\xi^{q}} \int_{\mcl O} u^{\alpha \rho} u^{p-1-\alpha \rho+\ell}\dif x \\
& \le \xi^{\rho(1+\beta+\delta)-q} \left(\int_{\mcl O} w^{\frac{2\theta}{1-\rho}}\dif x\right)^{1-\rho} \left(\frac{\overline{u^{\alpha}}}{\xi^{1+\beta+\delta}}\right)^\rho\\
& \le C   \xi^{\rho(1+\beta+\delta)-q} \left(\|\nabla w\|_{L^2}+\|w\|_{L^2}\right)^{2 \theta \gamma}  \|w\|_{L^2}^{2 \theta(1-\gamma)} \left(\frac{\overline{u^{\alpha}}}{\xi^{1+\beta+\delta}}\right)^\rho.
\end{split}
\Ee
Note that $\gamma \in (0,1)$, the above and
Young inequalities give
\Bes
\begin{split}
\frac{1}{\xi^q} \int_{\mcl O} u^{p-1+\ell}\dif x
& \le \theta \gamma c^{\frac 1{\gamma \theta}} \left(\|\nabla w\|_{L^2}+\|w\|_{L^2}\right)^2 \\
&\ \ \ +C \xi^{\frac{\rho(1+\beta+\delta)-q}{1-\theta \gamma}} \left(\frac{\overline{u^{\alpha}}}{\xi^{1+\beta+\delta}}\right)^{\frac{\rho}{1-\gamma \theta}}\|w\|^{\frac{2\theta(1-\gamma)}{1-\theta \gamma}}_{L^2},
\end{split}
\Ees
this, together with \eqref{e:WNorEqn}, yields that as $c$ is sufficiently small
\Be \label{e:EneInqW}
\begin{split}
\p_t \|w\|^2_{L^2} & \le C\xi^{\frac{\rho(1+\beta+\delta)-q}{1-\theta \gamma}} \left(\frac{\overline{u^{\alpha}}}{\xi^{1+\beta+\delta}}\right)^{\frac{\rho}{1-\gamma \theta}}\|w\|^{\frac{2\theta(1-\gamma)}{1-\theta \gamma}}_{L^2} \\
& \le C\left(\inf_{0 \le s \le 1}\xi(s)\right)^{\frac{\rho(1+\beta+\delta)-q}{1-\theta \gamma}} \left(\frac{\overline{u^{\alpha}}}{\xi^{1+\beta+\delta}}\right)^{\frac{\rho}{1-\gamma \theta}}\|w\|^{\frac{2\theta(1-\gamma)}{1-\theta \gamma}}_{L^2} \ \ \ \forall t \in [0,1],
\end{split}
\Ee
where the last inequality is by the fact $\frac{\rho(1+\beta+\delta)-q}{1-\theta \gamma}<0$ (due to the assumption of $\delta$).
%Let $T \in (0,1)$ be some number to be chosen later and write
Thanks to \eqref{e:XiGeInf}, we have
\Be \label{e:TheEst}
\left(\inf_{0 \le t \le 1}\xi(t)\right)^{\frac{\rho(1+\beta+\delta)-q}{1-\theta \gamma}} \le \Theta.
\Ee
Writing $\eta(t)=\|w(t)\|^2_{L^2},$
it follows from \eqref{e:EneInqW} and \eqref{e:TheEst} that
\Be
\begin{split}
\p_t \eta(t)
& \le C \Theta  \left(\sup_{0 \le t \le 1}\eta(t)\right)^{\frac{\theta(1-\gamma)}{1-\theta \gamma}} \left(\frac{\overline{u^{\alpha}}(t)}{\xi^{1+\beta+\delta}(t)}\right)^{\frac{\rho}{1-\gamma \theta}}  \ \ \ \ \forall t \in [0,1].
\end{split}
\Ee
Thanks to the second inequality in \eqref{e:RhoCon}, we have $\frac{\rho}{1-\gamma \theta}<1$ as
$\ell$ is sufficiently large, thus the above and H\"{o}lder inequalities give
\Be \label{e:SupBouU}
\begin{split}
\sup_{0 \le t \le 1} \eta(t)
 \le \eta(0)+C \Theta \left(\int_0^1 \frac{\overline{u^{\alpha}}(s)}{\xi^{1+\beta+\delta}(s)} \dif s\right)^{\frac{\rho}{1-\gamma \theta}} \left(\sup_{0 \le t \le 1}\eta(t)\right)^{\frac{\theta(1-\gamma)}{1-\theta \gamma}}
\end{split}
\Ee
If $\sup_{0 \le t \le 1} \eta(t)>1$, \eqref{e:SupBouU} implies
\Bes
\begin{split}
\left(\sup_{0 \le t \le 1} \eta(t)\right)^{\frac{1-\theta}{1-\theta \gamma}}\le \eta(0)+C\Theta \left(\int_0^1
\frac{\overline{u^{\alpha}}(s)}{\xi^{1+\beta+\delta}(s)} \dif s\right)^{\frac{\rho}{1-\gamma \theta}}
\end{split}
\Ees
and thus
\Bes
\begin{split}
\sup_{0 \le t \le 1} \eta(t) \le \eta^{\frac{1-\gamma \theta}{1-\theta}}(0)+C\Theta^{\frac{1-\theta \gamma}{1-\theta}} \left(\int_0^1 \frac{\overline{u^{\alpha}}(s)}{\xi^{1+\beta+\delta}(s)} \dif s\right)^{\frac{\rho}{1-\theta}}.
\end{split}
\Ees
This and Lemma \ref{l:IntBarUXi} give
\Bes
\begin{split}
\sup_{0 \le t \le 1} \eta(t) \le  C \left(\|v\|_{L^\ell}^{\frac{(1-\theta \gamma)\ell}{1-\theta}}+\Theta^{\frac{1-\theta \gamma}{1-\theta}} \Lambda^{\frac{\rho}{1-\theta}}(\delta,\zeta,B,\mcl M_\delta)\right) \ \ {\rm if} \  \sup_{0 \le t \le 1} \eta(t)>1.
\end{split}
\Ees
Combining this with the case $\sup_{0 \le t \le 1} \eta(t) \le 1$ immediately yields the
desired inequality.
\end{proof}

\subsection{Existence and uniqueness of the global solution}
Before proving the global existence and uniqueness of the solution, we recall some facts from (\cite[pp. 15-16]{kri08}). Take $\Delta$ with Neumman boundary
as an operator on $L^\theta(\mcl O)$ with $\theta \ge 1$, the associated Helmholtz operator is defined
$$\mcl H=I-\Delta,$$
we can define $\mcl H^\alpha$ for all $\alpha$ since $S(t)$ is an analytic operator. Define $D(\mcl H^\alpha_\theta)$ the
domain of $\mcl H^\alpha$ equipped with the norm $\|.\|_{D(\mcl H^\alpha_\theta)}=\|.\|_{L^\theta}+\|\mcl H^\alpha .\|_{L^\theta}$.
There exists some $t_0>0$ such that for all $t \in (0,t_0]$
\Be \label{e:L2StfEst}
\|\mcl H^\alpha S(t).\|_{D(\mcl H_\theta^\alpha)} \lesssim  t^{-\alpha} \|.\|_{L^\theta}.
\Ee
  As
$\alpha>\frac{d}{2\theta}$, $D(\mcl H^\alpha)$ is continuously embedded in $C(\mcl O)$

\begin{proof} [Proof of Theorem \ref{t:MaiSol}]
The properties of the solution is easy to get from the previous a'priori estimates. We shall concentrate on proving the global unique solution and follow the
spirit in \cite{kri08}.

By the a'priori estimates of \eqref{e:XiLeSup} and \eqref{e:XiGeInf}, to show the global existence of Eq. \eqref{e:SGM}, it suffices to show that
$u$ can be globally extended.
Suppose that there exists some measurable set $A \subset \Omega$ with $\PP(A)>0$ such that
for each $\omega \in A$ there exists some $T^{*}_{\omega}$ such that
$$\lim_{t \uparrow T^{*}_\omega} \|u(t)\|_{C}=\infty.$$
Without loss of generality, we may assume $T^*_\omega<1$. Without loss of generality, we assume that $T^*_\omega>t_0$ where $t_0$ is the constant in \eqref{e:L2StfEst}. Let $t^*=T^*_\omega-\frac{t_0}{2}$,
%we take a partition $0=t_0<t_1<...<t_k=T^*_\omega-\delta$ with
%$t_i-t_{i-1}=\frac{T^*_\omega-\delta}{k}$ for all $i=1,...,k$.
choosing $p$ such that $\frac{d}{2p}<1$ and some $\alpha \in (\frac{d}{2p},1)$, by \eqref{e:L2StfEst} and \eqref{e:XiGeInf}, for all
$t \in (t^*,T^*_\omega-\e]$ with any $\e \in (0,t_0/4)$ we have
 \Be \label{e:UniNabUt}
 \begin{split}
 \|u(t)\|_{{D(\mcl H_\theta^\alpha)}} & \le  \|S(t-t^*)u(t^*)\|_{{D(\mcl H_\theta^\alpha)}}+ \int_{t^*}^{t}  \left\|S(t-s)\frac{u(s)^p}{\xi(s)^q}\right\|_{{D(\mcl H_\theta^\alpha)}} \dif s \\
 & \lesssim  (t-t^*)^{-\alpha} \|u(t^*)\|_{L^\theta}+\int_{t^*}^{t}  (t-s)^{-\alpha} \frac{\|u(s)\|_{{L^{\theta p}}}^p}{\xi(s)^q} \dif s \\
 & \lesssim  (t-t^*)^{-\alpha} \|u(t^*)\|_{L^\theta}+{\rm e}^{\frac 32 q+q B^*_1} \zeta^{-q} (t-t^*)^{1-\alpha} \sup_{0 \le  s \le T^{*}_\omega-\e} \|u(s)\|_{{L^{\theta p}}}^p.
 \end{split}
 \Ee
where $\sup_{0 \le  s \le T_\omega^{*}-\e} \|u(s)\|_{{L^{\theta p}}} \le \tl C$ where $\tl C$ only depends on $v$, $\zeta$, $p, q, \theta, \alpha, \beta, \omega$ by Lemma \ref{l:ExtSol}.
Since $\e \in (0,t_0/4)$ and $t^*=T^*_\omega-\frac{t_0}{2}$, from the above inequality we get
$$\|u(T^*_\omega-\e)\|_{{D(\mcl H_\theta^\alpha)}} \lesssim  {t}^{-\alpha}_0 \|v\|_{L^\theta}+{\rm e}^{\frac 32 q t+q B^*_1} \zeta^{-q} {t}^{1-\alpha}_0 \tl C.$$
By the Sobolev embedding, we further get
$$\|u(T^*_\omega-\e)\|_{{C}} \lesssim  {t}^{-\alpha}_0 \|u(t^*)\|_{L^\theta}+{\rm e}^{\frac 32 q t+q B^*_1} \zeta^{-q} {t}^{1-\alpha}_0 \tl C.$$
Since $\e>0$ can be arbitrarily small, we have
$$\|u(T^*_\omega-)\|_{{C}} \lesssim  {t_0}^{-\alpha} \|u(t^*)\|_{L^\theta}+{\rm e}^{\frac 32 q t+q B^*_1} \zeta^{-q} {t_0}^{1-\alpha} \tl C.$$
Contradiction.
Hence, Eq. \eqref{e:SGM} admits a global unique solution for all $\omega \in \Omega$ a.s..
\end{proof}

\section{Large deviation results}

Now we recall the definition of the large deviation principle. Let $\{X^\e,\e>0\}$ be a family of random variables defined on a probability space
$(\Omega,\mathcal{F},\mathbb{P})$
and taking values in a Polish space $\mathcal{E}$. Denote expectation with
respect to $\mathbb{P}$ by $\mathbb{E}$.
The large deviation principle is concerned with exponential decay of $\mathbb{P}(X^\e\in O)\ as\ \e\rightarrow 0$.

    \begin{defn}\label{Dfn-Rate function}
       \emph{\textbf{(Rate function)}} A function $I: \mathcal{E}\rightarrow[0,\infty]$ is called a rate function on
       $\mathcal{E}$,
       if for each $M<\infty$ the level set $\{x\in\mathcal{E}:I(x)\leq M\}$ is a compact subset of $\mathcal{E}$.
       For
       $O\in \mathcal{B}(\mathcal{E})$,
       we define $I(O)\doteq\inf_{x\in O}I(x)$.
    \end{defn}
    \begin{defn}
       \emph{\textbf{(Large deviation principle)}} Let $I$ be a rate function on $\mathcal{E}$. The sequence
       $\{X^\e\}$
       is said to satisfy the large deviation principle on $\mathcal{E}$ with rate function $I$ if the following two
       conditions
       hold.

         a. Large deviation upper bound. For each closed subset $F$ of $\mathcal{E}$,
              $$
                \limsup_{\e\rightarrow 0}\e\log\mathbb{P}(X^\e\in F)\leq-I(F).
              $$

         b. Large deviation lower bound. For each open subset $G$ of $\mathcal{E}$,
              $$
                \limsup_{\e\rightarrow 0}\e\log\mathbb{P}(X^\e\in G)\geq-I(G).
              $$
    \end{defn}

\begin{rem}
Note that the $I$ above is a function from sets to real numbers. To define the rate function $I$, it suffices to define its value at each point.
\end{rem}

\subsection{Large deviation result and the method}

 Without loss of generality,
we shall prove the LDP result for the dynamics in the time interval $[0,1]$.
Before stating our large deviation result, let us first recall the following preliminary.

The Cameron-Martin space associated to the Brownian motion $B_t$ is as follows:
\Bes
H=\{h \in H^1([0,1];\R): h(t)=\int_0^t \dot{h}(s) \dif s, \ \|\dot{h}\|_{L^2([0,1],\R)}<\infty\}.
\Ees
$H$ is a Hilbert space with the norm
$$\|h\|_H=\|\dot{h}\|_{L^2([0,1],\R)} \ \ \ \ \forall h \in H.$$
It is clear to see
\Be \label{e:MaxhH}
|h(t)-h(s)| \le \|h\|_H \ \ \ \ \ \forall 0 \le s<t \le 1.
\Ee
Fix $N>0$, and denote $$\mathcal{A}^d_N=\{h\in H,\ \|h\|_H\leq N\}.$$
Then ${\mathcal{A}}^d_N$ is a compact Polish space endowed with the weak topology of $H$. Denote the weak convergence in ${\mathcal{A}}^d_N$ by $\cdot \rightharpoonup \cdot$, for $\{h_n\}_n \subset H$ and $h \in H$, $h_n \rightharpoonup h$ if
$$\lim_{n\rightarrow \infty} \int_0^1 \phi(s) \dot h_n(s) \dif s=\int_0^1 \phi(s) \dot h(s) \dif s \ \ \ \ \forall \ \phi \in L^2([0,1];\R).$$
Define
\Bes
\begin{split}
\mcl A^s=\{h;  \  h:\Omega \times [0,1]  \rightarrow \R \ &{\rm satisfies} \ \  h(\omega,.)\in H  \ \ \forall \omega \in \Omega \\
& {\rm and}
\  \ h(.,t) {\rm \ is \ } \mcl F_t \ {\rm measurable} \ \ \forall t  \in [0,1]\}
\end{split}
\Ees
and for all $N>0$
$$\mcl A^s_N=\{h \in \mcl A^s: \|h(\omega)\|_H \le N \ \ \ \ \forall \omega \in \Omega\}.$$
Let $h \in H$, consider the following differential equation
\Be \label{e:UhXiH}
\begin{split}
& \p_t u_h=\Delta u_h-u_h+\frac{u_h^p}{\xi_h^q}, \\
& \dif \xi_h=-\xi_h \dif t+\frac{\overline{u_h^\alpha}}{\xi_h^\beta}  \dif t+\xi_h \dif h(t), \\
\end{split}
\Ee
with the same boundary and initial conditions as in Eq. \eqref{e:SGM}.

Let $\e \in [0,1]$ and let $(h_\e)_{0 \le \e \le 1} \subset \mcl A^s$, to study the large deviation of Eq. \eqref{e:SGM}, we also need to consider the following stochastic PDEs:
\Be \label{e:UhXiEh}
\begin{split}
& \p_t u_{\e,h_\e}=\Delta u_{\e,h_\e}-u_{\e,h_\e}+\frac{u_{\e,h_\e}^p}{\xi_{\e,h_\e}^q}, \\
& \dif \xi_{\e,h_\e}=-\xi_{\e,h_\e} \dif t+\frac{\overline{u_{\e,h_\e}^\alpha}}{\xi_{\e,h_\e}^\beta}  \dif t+\sqrt{\e} \xi_{\e,h_\e} \dif B_t+\xi_{\e,h_\e} \dif h_\e(t), \\
\end{split}
\Ee
with the same boundary and initial conditions as in Eq. \eqref{e:SGM}. By the same argument as in the previous section, we can prove the global existence and uniqueness of the solutions to Eqs.
\eqref{e:UhXiH} and \eqref{e:UhXiEh}. \\

Now we are at the position to state our large deviation result.
\begin{thm}[Large deviation principle]\label{T:LDP}
Let $\{(u_\e,\xi_\e)\}$ be the solution of the equation
\begin{equation} \label{e:LDP-SGM}
\begin{cases}
& \p_t u_\e=\Delta u_\e-u_\e+\frac{u_\e^p}{\xi_\e^q}, \\
& \dif \xi_\e=-\xi_\e \dif t+\frac{\overline{u_\e^\alpha}}{\xi_\e^\beta}  \dif t+\sqrt{\e} \xi_\e \dif B_t, \\
& \frac{\p u_\e}{\p \nu}=0, \\
& u_\e(0)=v, \\
& \xi_\e(0)=\zeta.
\end{cases}
\end{equation}
Then $\{(u_\e,\xi_\e)\}$ satisfies a large deviation principle in $ C([0,1];C\times \mathbb{R})$ with the rate function $I$ given by: for any $(u,\xi)\in C([0,1];C\times \mathbb{R})$,
$$
I((u,\xi)):=\inf_{\{h \in H: (u_h,\xi_h)=(u,\xi)\}}\left(\frac{1}{2}\|h\|_H^2\right),
$$
with the convention $\inf\{\emptyset\}=\infty$, where $(u_h, \xi_h)$ is the solution to Eq. \eqref{e:UhXiH}.
\end{thm}
We shall follow the method in \cite[Theorem 4.4]{BD} to prove the above LDP. According to this method, we only need to show the following two propositions.
\begin{proposition} \label{p:LDP2}
Let $g_n,\ h \in {\mathcal{A}}^d_N$ and $(u_{g_n},\xi_{g_n})$ be the solution of Eq. \eqref{e:UhXiH} with $h$ replaced by ${g_n}$. Up to taking a
subsequence, we have
$$\lim_{g_n \rightharpoonup h} \|(u_{g_n},\xi_{g_n})-(u_h, \xi_h)\|_{C([0,1];C \times \R)}=0.$$
\end{proposition}

\begin{proposition}\label{p:LDP3}
For a family $\{h_\e\}\subset \mathcal{A}^s_N$ for which $h_\e$ converges in distribution
to $h$ under the weak topology of $H$, up to taking a subsequence, the solution $(u_{\e,h_\e},\xi_{\e,h_\e})$ of (\ref{e:UhXiEh}) converges in distribution to $(u_{h},\xi_{h})$, more precisely, for all bounded continuous function $f: C([0,1];C \times \R) \rightarrow \R$, up to taking a subsequence, the following relation holds:
\Be
\lim_{\e \rightarrow 0} \E f(u_{\e,h_\e},\xi_{\e,h_\e})=\E f(u_{h},\xi_{h}).
\Ee
\end{proposition}
\ \ \ \ \ \ \
\subsection{Proof of Proposition \ref{p:LDP2}}
\begin{lem}\label{l:4.2}
For all $t \in [0,1]$,
we have the following estimates
\Be \label{e:XihGe}
\begin{split}
& \xi_h(t) \ge \re^{-t-\|h\|_H} \zeta,
\end{split}
\Ee
\Be \label{e:XihLe}
\begin{split}
\xi_h(t) \lesssim \re^{\|h\|_H} \zeta+\re^{\|h\|_H} \left(
 \sup_{0 \le t \le 1} \overline{u_h^{\alpha}}(t)\right)^{\frac{1}{1+\beta}}.
\end{split}
\Ee
\end{lem}

\begin{proof}
From Eq. \eqref{e:UhXiH}, we have
\Be
\begin{split}
\dif \xi_h^{1+\beta}(t)=-(1+\beta) \xi_h^{1+\beta}(t) \dif t+(1+\beta) \xi_h^{1+\beta}(t) \dif h(t)+(1+\beta)\overline{u_h^{\alpha}}(t) \dif t,
\end{split}
\Ee
which clearly implies
\Bes
\begin{split}
\xi_h^{1+\beta}(t)&=\re^{-(1+\beta)t+(1+\beta)h(t)} \zeta^{1+\beta}+(1+\beta)\int_0^t
\re^{-(1+\beta)(t-s)+(1+\beta)(h(t)-h(s))}  \overline{u_h^{\alpha}}(s) \dif s.
\end{split}
\Ees
This equality and \eqref{e:MaxhH} clearly imply the desired two inequalities.
\end{proof}

\begin{lem}  \label{l:IntUhXihEst}
We have
$$\int_0^t \frac{\overline{u_h^{\alpha}}(s)}{\xi_h^{1+\delta+\beta}(s)} \dif s \le \Lambda(\delta,\zeta,h) \ \ \ \forall t \in [0,1],$$
where
%$$\Lambda(\zeta,h)=\frac{\zeta+\re^{1+\|h\|_H}}{\zeta} \left[\frac 1{\log(1+\re^{-1-\|h\|_H} \zeta)}+\frac 1{\log(1+\zeta)}+\|h\|_H+1\right].$$
$$
\Lambda(\delta,\zeta,h)=\delta^{-1} \zeta^{-\delta}+\re^{\delta(1+\|h\|_H)} \zeta^{-\delta}+\re^{\delta(1+\|h\|_H)} \|h\|_H.
$$

\end{lem}

\begin{proof} Differentiating $\xi_h^{-\delta}(t)$ we get
\Bes
\begin{split}
\xi_h^{-\delta}(t)-\zeta^{-\delta}=\delta \int_0^t \xi_h^{-\delta}(s) \dif s-\delta \int_0^t \frac{\overline{u_h^{\alpha}}(s)}{\xi_h^{1+\delta+\beta}(s)} \dif s-\delta \int_0^t \xi_h^{-\delta}(s) \dif h_s,
\end{split}
\Ees
which, together with \eqref{e:XihGe} and H\"{o}lder inequality, gives
\Bes
\begin{split}
\int_0^t \frac{\overline{u_h^{\alpha}}(s)}{\xi_h^{1+\delta+\beta}(s)} \dif s
& \le \delta^{-1} \zeta^{-\delta}+\int_0^t \xi_h^{-\delta}(s) \dif s+\left|\int_0^t \xi_h^{-\delta}(s) \dot h_s \dif s\right| \\
& \le \delta^{-1} \zeta^{-\delta}+\re^{\delta(1+\|h\|_H)} \zeta^{-\delta}+\left(\int_0^t \xi_h^{-2\delta}(s)\dif s\right)^{\frac12} \|h\|_H \\
& \le \delta^{-1} \zeta^{-\delta}+\re^{\delta(1+\|h\|_H)} \zeta^{-\delta}+\re^{\delta(1+\|h\|_H)} \|h\|_H
\end{split}
\Ees
for all $t \in [0,1]$.
This completes the proof.
\end{proof}

\begin{lem} \label{l:LDP1}
Let $\rho, \ell,\theta, \gamma$ be the same as those in Lemma \ref{l:ExtSol}. Let $\delta \in (0, \frac{q-\rho-\rho \beta}{\rho})$.
As $\ell$ is sufficiently large so that $\theta \in (0,1)$, $\gamma \in (0,1)$ and $\frac{\rho}{1-\gamma \theta} \in (0,1)$, we have
\Bes
\begin{split}
& \sup_{0\le t \le 1} \|u_h(t)\|^\ell_{L^\ell} \le C \left(\|v\|_{L^\ell}^{\frac{(1-\theta \gamma)\ell}{1-\theta}}+\tl \Theta^{\frac{1-\theta \gamma}{1-\theta}} \Lambda^{\frac{\rho}{1-\theta}}(\delta,\zeta,h))\right) \vee 1.
\end{split}
\Ees
where $C$ depends on $\alpha, \beta, p, q$,
%$\theta=\frac{1}{\ell}(p-1-\alpha \rho+\ell)$, $\gamma=\frac{d(\rho+\theta-1)}{2 \theta}$, $\rho$ satisfies the condition \eqref{e:RhoCon},
$\Lambda(\delta, \zeta, h)$ is defined in Lemma \ref{l:IntUhXihEst} and
$$\tl \Theta=\re^{\frac{q-\rho(1+\beta+\delta)}{1-\theta \gamma}}
\zeta^{\frac{\rho(1+\beta+\delta)-q}{1-\theta \gamma}} \re^{\frac{q-\rho(1+\beta+\delta)}{1-\theta \gamma}\|h\|_H}.$$
\end{lem}

\begin{proof}
Repeating the argument for deriving \eqref{e:SupBouU} and using \eqref{e:MaxhH}, we get
\Bes
\begin{split}
\sup_{0 \le t \le 1} \eta(t)
 \le \eta(0)+C \tl \Theta \left(\int_0^1 \frac{\overline{u_h^{\alpha}}(s)}{\xi_h^{1+\beta+\delta}(s)} \dif s\right)^{\frac{\rho}{1-\gamma \theta}} \left(\sup_{0 \le t \le 1}\eta(t)\right)^{\frac{\theta(1-\gamma)}{1-\theta \gamma}},
\end{split}
\Ees
where $\eta(t)=\|u_h(t)\|^\ell_{L^\ell}$. By the same argument as that below \eqref{e:SupBouU}, we get the
desired inequality.
\end{proof}
\vskip 3mm

\begin{lem} \label{l:sup-u}
 Let $(u_h,\xi_h)$ be the solution of Eq. \eqref{e:UhXiH}. We have
%\begin{eqnarray}\label{e:sup-u}
%\sup_{h\in \overline{\mathcal{A}}_N}\|u_h\|_{C([0,1];C)}\leq C_N.
%\end{eqnarray}
%and
\begin{eqnarray}\label{e:sup-u-xi}
\sup_{h\in {\mathcal{A}}^d_N}\|(u_h,\xi_h)\|_{C([0,1];C\times \mathbb{R})}\leq C
\end{eqnarray}
where $C$ depends on $N, \zeta, \|v\|_C, \alpha, \beta, p, q$.
\end{lem}
\begin{proof}
Similar as in the proof of Lemma \ref{l:LocSol}, set
\Bes
\begin{split}
\mathcal{A}_{T,M,N}=\bigg\{(u,\xi)\in C([0,T];C(\mathcal{O},\mathbb{R})& \times\mathbb{R}): u(t)\geq0,\ \xi(t)\geq \re^{-t-N}\zeta,\forall0\leq t\leq T;\\
                   &  u(0)=v,\ \xi(0)=\zeta;\ \|(u,\xi)\|_{C([0,T];C\times\mathbb{R})}\leq M \bigg\}
\end{split}
\Ees
with $M>2+\|v\|_C+\re^N\zeta$ and $T>0$ being some number depending on $N,M, \alpha,\beta, p, q$. By a similar argument as in the proof of
Lemma \ref{l:LocSol}, we have
\begin{eqnarray}\label{e:1}
\sup_{h\in {\mathcal{A}}^d_N}\|(u_h,\xi_h)\|_{C([0,T];C\times\mathbb{R})}\leq M.
\end{eqnarray}

\noindent To complete the proof, we only need to bound the solution on the time interval $[T,1]$. On the one hand,
by \eqref{e:XihLe}, \eqref{e:XihGe} and Lemma \ref{l:LDP1}, there exists some $\bar C$ depending only on $v, \zeta, N$ such that
\begin{eqnarray}\label{e:4}
\sup_{h\in{\mathcal{A}}^d_N}\|\xi_h\|_{C([0,1];\mathbb{R})}\leq \overline{C}.
\end{eqnarray}
Repeating the argument in the proof of Theorem \ref{t:MaiSol} and choosing $\alpha>\frac{d}{2 \theta}$, we have some $\hat C$ depending only
on $v, \zeta, \alpha, \beta, N$ such that
$$
\sup_{h\in{\mathcal{A}}^d_N}\sup_{T\leq t\leq1}\|u_h\|_{D(\mcl H^\alpha_p)}\leq \hat C.
$$
This and Sobolev embedding theorem further give
\begin{eqnarray}\label{e:2}
\sup_{h\in{\mathcal{A}}^d_N}\|u_h\|_{C([T/2,1];C)}\leq \widetilde{C}
\end{eqnarray}
where $\tl C$ depends only
on $v, \zeta, \alpha, \beta, N$.
Hence,
\begin{eqnarray}\label{e:3}
\sup_{h\in{\mathcal{A}}^d_N}\|(u_h, \xi_h)\|_{C([0,1];C \times \R)}\leq \tl C+\bar C.
\end{eqnarray}
The proof is complete.
\end{proof}

\begin{proof} [Proof of Proposition \ref{p:LDP2}]
Let all $C$ below be some numbers depending on $N, \zeta, \|v\|_C, \alpha, \beta, p, q$, whose exact values may vary from line to line.
Recall $S(t)=e^{(\Delta-1)t}$ and denote $\Lambda_{n,m}(t)=u_{g_n}(t)-u_{g_m}(t)$. Observe
$$
\Lambda_{n,m}(t)=\int_0^tS(t-s)\left(\frac{u^p_{g_n}(s)}{\xi^q_{g_n}(s)}-\frac{u^p_{g_m}(s)}{\xi^q_{g_m}(s)}\right)\dif s.
$$
Thanks to Lemma \ref{l:4.2} and Lemma \ref{l:sup-u}, we have
\begin{equation} \label{e:Lambda1}
\begin{split}
\|\Lambda_{n,m}(t)\|_C
&\leq \int_0^t \left\|\frac{u^p_{g_n}(s)}{\xi^q_{g_n}(s)}-\frac{u^p_{g_m}(s)}{\xi^q_{g_m}(s)}\right\|_C \dif s \\%\nonumber\\
&\leq
\int_0^t\frac{\left\|u^p_{g_n}(s)-u^p_{g_m}(s)\right\|_C}{\xi^q_{g_n}(s)}\dif s+\int_0^t\left\|u_{g_m}(s)\right\|^p_C
\left|\frac{1}{\xi^q_{g_n}(s)}-\frac{1}{\xi^q_{g_m}(s)}\right|\dif s \\%\nonumber\\
& \le C\int_0^t \Lambda_{m,n}(s) \dif s+C \int_0^t|\xi_{g_n}(s)-\xi_{g_m}(s)|\dif s.
\end{split}
\end{equation}

\noindent For all $s,t \in [0,1]$ and $g_n \in \mcl A^d_N$, by Lemma \ref{l:sup-u} and the second equation of \eqref{e:UhXiH}, we have
\begin{equation*}
\begin{split}
   |\xi_{g_n}(t)-\xi_{g_n}(s)|
&\leq
   \int_s^t \xi_{g_n}(r)\dif r+\int_s^t \frac{\overline{u^\alpha_{g_n}}(r)}{\xi_{g_n}^\beta(r)} \dif r+\int_s^t \xi_{g_n}(r)\left|\dot{g}_n(r)\right|
   \dif r\\
&\leq
   C(t-s)+C(t-s)+\left(\int_s^t|\xi_{g_n}(r)|^2\dif r\right)^{\frac 12}\left(\int_0^1|\dot{g}_n(r)|^2\dif r\right)^{\frac 12}\\
&\leq
      C(t-s)+C (t-s)^{\frac 12}.
\end{split}
\end{equation*}
 The above inequality clearly implies that $\{\xi_{g_n},\ n\geq1\}$ is equi-continuous. By $\rm Arzel\grave{a}$-Ascoli Theorem, there exist some
$\xi\in C([0,1];\mathbb{R})$ and a subsequence of $\{\xi_{g_n},\ n\geq1\}$ (say $\{\xi_{g_n},\ n\geq1\}$ without loss of generality) such that
\begin{eqnarray}\label{e:xi}
\lim_{n\rightarrow\infty}\|\xi_{g_n}-\xi\|_{C([0,1];\mathbb{R})}=0.
\end{eqnarray}
It follows from (\ref{e:XihGe}) and \eqref{e:sup-u-xi} that for all $t \in [0,1]$
$$\xi(t) \ge \re^{-t-\|h\|_H} \zeta.$$
Moreover, \eqref{e:xi} and \eqref{e:Lambda1} clearly imply that $\{u_{g_n},\ n\geq1\}$ is a Cauchy sequence in $C([0,1];C)$. Hence, there exists some $u\in C([0,1];C)$ so that
\Be \label{e:Un-UCon}
\lim_{n\rightarrow\infty}\|u_{g_n}-u\|_{C([0,1];C)}=0.
\Ee

\noindent Since
$$
u_{g_n}(t)=S(t)v+\int_0^tS(t-s)\frac{u^p_{g_n}(s)}{\xi^q_{g_n}(s)} \dif s,
$$
letting $n \rightarrow \infty$ we get
\begin{eqnarray}\label{e:u}
u(t)=S(t)v+\int_0^tS(t-s)\frac{u^p(s)}{\xi^q(s)}\dif s.
\end{eqnarray}
On the other hand, by \eqref{e:xi} and $g_n \rightharpoonup h$ in $H$,
\begin{equation*}
\begin{split}
& \ \ \ \ \int_0^t\xi_{g_n}(s)\dot{g}_n(s) \dif s-\int_0^t\xi(s)\dot{h}(s) \dif s\\
&=\int_0^t[\xi_{g_n}(s)-\xi(s)]\dot{g}_n(s)\dif s+\int_0^t\xi(s)[\dot{g}_n(s)-\dot{h}(s)] \dif s
\rightarrow 0
\end{split}
\end{equation*}
as $n \rightarrow \infty$. Let $n \rightarrow \infty$, the above limit and the following relation
%\lim_{n\rightarrow\infty}\Big[\xi_{g_n}(t)-\zeta+\int_0^t \xi_{g_n}(s) \dif s-\int_0^t \frac{\overline{u_{g_n}^\alpha}(s)}{\xi_{g_n}^\beta(s)} \dif s\Big]\\
%=\xi(t)-\zeta+\int_0^t \xi(s) \dif s-\int_0^t \frac{\overline{u^\alpha}(s)}{\xi^\beta(s)} \dif s,
$$
\xi_{g_n}(t)=\zeta-\int_0^t \xi_{g_n}(s) \dif s+\int_0^t \frac{\overline{u_{g_n}^\alpha}(s)}{\xi_{g_n}^\beta(s)} \dif s+\int_0^t\xi_{g_n}(s)\dot{g}_n(s)\dif s
$$
give
\Bes
\xi(t)=\zeta-\int_0^t \xi(s) \dif s+\int_0^t \frac{\overline{u^\alpha}(s)}{\xi^\beta(s)} \dif s+\int_0^t\xi(s)\dot{h}(s)\dif s
\Ees
which, together with \eqref{e:u}, implies that $(u,\xi)$ solve Eq. \eqref{e:UhXiH}. Thanks to the uniqueness, we have $(u,\xi)=(u_h,\xi_h)$ and
thus
$$\lim_{g_n \rightharpoonup h} \|(u_{g_n},\xi_{g_n})-(u_h, \xi_h)\|_{C([0,1];C \times \R)}=0.$$
\end{proof}

%Let $\Lambda^n_2(t)=\xi_{g_n}(t)-\xi_{g_m}(t)$, we have
%\begin{eqnarray}\label{e:Lambda2}
%\Lambda^n_2(t)
%&=&
%  \int_0^t\widetilde{R}(t-s,g_t-g_s)\Big(\frac{\overline{u^\alpha_{g_n}}(s)}{\xi^\beta_{g_n}(s)}-\frac{\overline{u^\alpha_h}(s)}{\xi^\beta_h(s)}\Big)ds\nonumber\\
% & &+
%  \int_0^t\widetilde{R}(t-s,g_t-g_s)\xi_h(s)(\dot{g}(s)-\dot{h}(s))ds\nonumber\\
%&=&
%  I_1(t)+I_2(t).
%\end{eqnarray}
%Using similar arguments as (\ref{e:Lambda1}),
%\begin{eqnarray}\label{e:I-1}
%|I_1(t)|
%&\leq&
%    \alpha e^{N+(1+N)\beta}\zeta^{-\beta}C_N^{\alpha-1}\int_0^t\|u_{g_n}(s)-u_h(s)\|_Cds\nonumber\\
%&  &+
%   \beta e^{N+(1+N)(\beta+1)}\zeta^{-(1+\beta)}C_N^\alpha\int_0^t|\xi_{g_n}(s)-\xi_h(s)|ds,
%\end{eqnarray}
%and by $H\ddot{o}lder$ Inequality
%\begin{eqnarray}\label{e:I-2}
%      |I_2(t)|
%&\leq&
%     e^N \sup_{0\leq s\leq 1}|\xi_h(s)|[\int_0^t|\dot{g}(s)-\dot{h}(s)|ds]^{1/2}.
%\end{eqnarray}
%
%Combining (\ref{e:Lambda1}), (\ref{e:Lambda2}), (\ref{e:I-1}) and (\ref{e:I-2}), by Grownwall inequality,
%it is easy to see
\subsection{Proof of Proposition \ref{p:LDP3}}
\begin{lem}
Let $\e>0$ be such that $2-\beta \e>0$ and $h_\e \in \mcl A^s_N$.
We have the following estimates
\Be \label{e:XiehGe}
\begin{split}
& \xi_{\e,h_\e}(t) \ge \re^{-\frac{(2-\e\beta)}2 t-N+\sqrt{\e} B_t} \zeta \ \ \ \forall t \in [0,1],
\end{split}
\Ee
\Be \label{e:XiehGeInf}
\begin{split}
&\inf_{0 \le t \le 1} \xi_{\e,h_\e}(t) \ge \re^{-1-N-\sqrt{\e} B^{*}_1} \zeta.
\end{split}
\Ee
Moreover, we have
\Be \label{e:XiehLeSup}
\begin{split}
&\sup_{0 \le t \le 1}\xi_{\e,h_\e}(t) \lesssim \re^{N+\sqrt{\e} B^{*}_1}\zeta^{1+\beta}+\re^{N+2\sqrt{\e} B^{*}_1}\left(\sup_{0 \le t \le 1}\overline{u_{\e,h_\e}^\alpha}(t)\right)^{\frac{1}{1+\beta}}.
\end{split}
\Ee
\end{lem}

\begin{proof}
We simply write $u=u_{\e, h_\e}$,  $\xi=\xi_{\e, h_\e}$ and $h=h_\e$.
By It$\hat{o}$ formula, we have
\Be
\begin{split}
\dif \xi^{1+\beta}(t)&=-\frac 12(1+\beta)(2-\e \beta) \xi^{1+\beta}(t) \dif t+(1+\beta) \xi^{1+\beta}(t) \dif h(t) \\
&\ \ \ \ +\sqrt{\e} (1+\beta) \xi^{1+\beta}(t)\dif B_t +(1+\beta)\overline{u^{\alpha}}(t) \dif t,
\end{split}
\Ee
which clearly implies
\Bes
\begin{split}
\xi^{1+\beta}(t)&=\re^{-\frac{(1+\beta)(2-\e\beta)}2 t+(1+\beta)h(t)+\sqrt{\e} (1+\beta) B_t} \zeta^{1+\beta} \\
&+(1+\beta)\int_0^t
\re^{-\frac{(1+\beta)(2-\e\beta)}2 (t-s)+(1+\beta)(h(t)-h(s))+\sqrt{\e} (1+\beta) (B_t-B_s)}  \overline{u^{\alpha}}(s) \dif s \\
& \le \re^{(1+\beta)\|h\|_H+\sqrt{\e} (1+\beta) B^*_1} \zeta^{1+\beta}+(1+\beta)\int_0^t
\re^{(1+\beta)\|h\|_H+2\sqrt{\e} (1+\beta) B^*_1}  \overline{u^{\alpha}}(s) \dif s,
\end{split}
\Ees
where the last inequality is by \eqref{e:MaxhH}. The above inequality clearly implies the three desired inequalities.
\end{proof}

Let $\delta>0$, define
$$\mcl M_{\e,\delta}(t)=\int_0^t \xi_{\e,h_\e}^{-\delta}(s) \dif B_s, \ \ \ \ \mcl M^{*}_{\e,\delta}=\sup_{0\le t \le 1}|\mcl M_{\e,\delta}(t)|.$$
\begin{lem} \label{e:M*BDG}
Let $\mu>0$ and $\delta>0$, for all $\e \in [0,1]$ we have
\Be
\begin{split}
\E \left(\mcl M^{*}_{\e,\delta}\right)^\mu \le C
\end{split}
\Ee
where $C$ depends only on $\mu, N, \delta$ and $\zeta$. Moreover, we have
\Be
\mcl M^{*}_{\e,\delta}<\infty \ \ \ \ \ a.s..
\Ee
\end{lem}
\begin{proof}
We only have to show the desired inequality for the case $\mu>2$ since the case of $0<\mu \le 2$ is
an immediate corollary from the former. We simply write $\xi_\e=\xi_{\e,h_\e}$.

By Burkholder-Davis-Gundy inequality and H\"{o}lder inequality, we have
\Bes
\begin{split}
\E \left(\mcl M^{*}_{\e,\delta}\right)^\mu & \le C\E \left[\int_0^1 \xi_{\e}^{-2\delta}(s) \dif s\right]^{\frac {\mu} 2}
\le C\left[\int_0^1 \E \xi_{\e}^{-\mu \delta}(s) \dif s\right].
\end{split}
\Ees
which, together with \eqref{e:XiehGeInf}, further gives
\Bes
\begin{split}
\E \left(\mcl M^{*}_{\e,\delta}\right)^\mu  \le C\E \re^{\mu \delta+\mu \delta N+\mu \delta \sqrt{\e} B^{*}_1} \zeta^{-\mu \delta}.
\end{split}
\Ees
The desired inequality immediately follows from the above inequality and \eqref{e:DenB*t}. The second inequality is a direct corollary from the first one.
\end{proof}

\begin{lem}  \label{l:IntUhXiheEst}
Let $\e>0$ be such that $2-\beta \e>0$ and let $h_\e \in \mcl A^s_N$. For all $\delta>0$, we have
$$\int_0^1 \frac{\overline{u_{\e,h_\e}^{\alpha}}(s)}{\xi_{\e,h_\e}^{1+\beta+\delta}(s)} \dif s \le \Lambda(\zeta,\e, B,N,\delta,\mcl M^*_{\e,\delta}),$$
where
\Bes
\begin{split}
\Lambda(\zeta,\e, B,N,\delta,\mcl M^*_{\e,\delta})=
\delta^{-1} \zeta^{-\delta}&+\frac{(2+\e+\delta\e+2 N)\re^{\delta+\delta N+\delta \sqrt{\e} B^{*}_1}}2 \zeta^{-\delta}+\sqrt{\e} \mcl M^*_{\e,\delta}
\end{split}
\Ees
\end{lem}
\begin{proof}
For the notational simplicity, we shall write $\xi(t)=\xi_{\e, h_\e}(t)$ and $u(t)=u_{\e,h_\e}(t)$.
Applying It$\hat{o}$ formula to $\xi^{-\delta}(t)$, we get
\Bes
\begin{split}
\xi^{-\delta}(t)-\zeta^{-\delta}=&\frac{\delta\left(2+\e+\delta\e \right)}2\int_0^t \xi^{-\delta}(s) \dif s-\delta \int_0^t \frac{\overline{u^{\alpha}}(s)}{\xi^{1+\delta+\beta}(s)} \dif s \\
&-\delta \int_0^t \xi^{-\delta}(s) \dot h_s \dif s-\delta \sqrt{\e} \int_0^t \xi^{-\delta}(s) \dif B_s,
\end{split}
\Ees
which gives
\Bes
\begin{split}
\int_0^t \frac{\overline{u^{\alpha}}(s)}{\xi^{1+\delta+\beta}(s)} \dif s
& \le \delta^{-1} \zeta^{-\delta}+\frac{2+\e+\delta\e}2 \int_0^t \xi^{-\delta}(s) \dif s \\
&+\sup_{0 \le t \le 1} \left|\int_0^t \xi^{-2\delta}(s) \dif s\right|^{\frac 12}\|h\|_H+\sqrt{\e} \sup_{0 \le t \le 1}\left|\int_0^t \xi^{-\delta}(s) \dif B_s\right|\\
& \le \delta^{-1} \zeta^{-\delta}+\frac{(2+\e+\delta\e+2 \|h_\e\|_H)\re^{\delta+\delta \|h_\e\|_H+\delta \sqrt{\e} B^{*}_1}}2 \zeta^{-\delta} \\
&
+ \sqrt{\e}\sup_{0 \le t \le 1}\left|\int_0^t \xi^{-\delta}(s) \dif B_s\right|
\end{split}
\Ees
where the last inequality is by \eqref{e:XiehGeInf}. This clearly implies the desired inequality.
\end{proof}

\begin{lem} \label{l:LDPUe}
 Let $\rho, \ell,\theta, \gamma$ be the same as those in Lemma \ref{l:ExtSol}. Let $h_\e \in \mcl A^s_N$ and $\delta \in (0, \frac{q-\rho-\rho \beta}{\rho})$.
As $\ell$ is sufficiently large so that $\theta \in (0,1)$, $\gamma \in (0,1)$ and $\frac{\rho}{1-\gamma \theta} \in (0,1)$, we have
 \Bes
 \sup_{0\le t \le 1} \|u_{\e,h_\e}(t)\|^\ell_{L^\ell} \le C \left(\|v\|_{L^\ell}^{\frac{(1-\theta \gamma)\ell}{1-\theta}}+\hat \Theta^{\frac{1-\theta \gamma}{1-\theta}} \Lambda^{\frac{\rho}{1-\theta}}(\zeta,\e, B,N,\delta,\mcl M^*_{\e,\delta})\right) \vee 1,
 \Ees
 where $C$ depends on $\alpha, \beta, p, q$,
 $\Lambda(\zeta, \e, B,N,\delta,\mcl M^*_{\e,\delta})$ is defined in Lemma \ref{l:IntUhXiheEst} and
$$\hat \Theta=
\re^{(1+N)\frac{q-\rho(1+\beta+\delta)}{1-\theta \gamma}} \zeta^{\frac{\rho(1+\beta+\delta)-q}{1-\theta \gamma}} \re^{\frac{q-\rho(1+\beta+\delta)}{1-\theta \gamma}\sqrt{\e} B^*_1}.$$
\end{lem}
\begin{proof}
Repeating the argument for getting \eqref{e:SupBouU} and using \eqref{e:MaxhH}, we get
\Be \label{e:SupBouSto}
\begin{split}
\sup_{0 \le t \le 1} \eta(t)
 \le \eta(0)+C \hat \Theta \left(\int_0^1 \frac{\overline{u_{\e,h_\e}^{\alpha}}(s)}{\xi_{\e,h_\e}^{1+\beta+\delta}(s)} \dif s\right)^{\frac{\rho}{1-\gamma \theta}} \left(\sup_{0 \le t \le 1}\eta(t)\right)^{\frac{\theta(1-\gamma)}{1-\theta \gamma}},
\end{split}
\Ee
where $\eta(t)=\|u_{\e,h_\e}(t)\|^\ell_{L^\ell}$. Repeating the argument below \eqref{e:SupBouU}, we immediately get the desired inequality.
\end{proof}
\ \ \ \

\begin{proof} [Proof of Proposition \ref{p:LDP3}]
For the notational simplicity, we shall write $u_\e=u_{\e,h_\e}$ and $\xi_\e=\xi_{\e,h_\e}$.
We choose $\ell>0$ in Lemma \ref{l:LDPUe} be sufficiently large so that $\ell>2\alpha$ and fix it. We also fix the number $\rho,\theta, \gamma, \delta$ in Lemma \ref{l:LDPUe}.
By their definitions, $\ell,\rho,\theta, \gamma, \delta$ are all some fixed numbers depending on $\alpha, \beta, p, q$.
Let all $C$ below be some numbers depending on $\zeta, v, \alpha, \beta, p, q$ and $N$, whose exact values may vary from one to one.
We shall prove the proposition by the following two steps.
\vskip 3mm

{\bf (Step 1)} We shall prove in Step 2 below that
there exist some $\xi \in C([0,1],\R)$ and a subsequence $\{\xi_{\e_n}\}$ with $\lim_{n \rightarrow \infty} \e_n=0$ such that
\Be \label{e:SubSeqCon}
\lim_{n \rightarrow \infty} \xi_{\e_n}=\xi \ \ \  {\rm \ in \ distribution \ under \ the \ topology } \ C([0,1],\R).
 \Ee
 By Skorohod embedding theorem, there exist a probability space
$(\hat \Omega, \hat{\mcl F}, \hat{\PP})$ and random variables $\{\hat{\xi}_{\e_n}\}$ and $\hat{\xi}$ which have the same distributions as
 $\{{\xi}_{\e_n}\}$ and ${\xi}$ respectively, such that
$$\lim_{n \rightarrow \infty} \|\hat{\xi}_{\e_n}-\hat{\xi}\|_{C([0,1];\R)}=0 \ \ \ \ \ a.s..$$

\noindent Consider the equations
\Bes
\p_t \hat u_{\e_n}=\Delta \hat u_{\e_n}-\hat u_{\e_n}+\frac{\hat u_{\e_n}^p}{\hat \xi_{\e_n}^q}, \ \ \ \hat u_{\e_n}(0)=v,
\Ees
\Be \label{e:BarUEqn}
\p_t \hat u=\Delta \hat u-\hat u+\frac{\hat u^p}{\hat \xi^q}, \ \ \ \hat u(0)=v,
\Ee
both with the same boundary condition, by the same argument as in the proof of Proposition \ref{p:LDP2}, we get
\Be
\lim_{n \rightarrow \infty}\|\hat u_{\e_n}-\hat u\|_{C([0,1];C)}=0 \ \ \ \ a.s..
\Ee
It is clear that the distribution of $(\hat u_{\e_n}, \hat{\xi}_{\e_n})$ is the same as those of $(u_{\e_n}, {\xi}_{\e_n})$.
By  \eqref{e:MarTerm} below, we have
$$\lim_{\e \rightarrow 0+}\E \sqrt{\e} \sup_{0 \le t \le 1} \left|\int_0^t \xi_\e \dif B_s\right|=0.$$
Hence, up to taking a subsequence, we have
$$\lim_{n \rightarrow \infty} \sqrt{\e_n}  \sup_{0 \le t \le 1} \left|\int_0^t \xi_{\e_n} \dif B_s\right|=0.$$
By the same argument as in the proof of Proposition \ref{p:LDP2}, we get
\Be \label{e:BarXiEqn}
\hat \xi(t)=\zeta-\int_0^t \hat \xi(s) \dif s+\int_0^t \frac{\overline{\hat u^\alpha}(s)}{\hat{\xi}^\beta(s)} \dif s+\int_0^t\hat \xi(s)\dot{h}(s)\dif s.
\Ee
\eqref{e:BarUEqn} and \eqref{e:BarXiEqn} yield that $(\hat u,\hat \xi)$ satisfies Eq. \eqref{e:UhXiH}. By uniqueness of the solution, $(\hat u,\hat \xi)$
and $(u_h,\xi_h)$ have the same distribution.
Hence, we complete the proof up to showing \eqref{e:SubSeqCon}. \\

{\bf (Step 2)} Now we show \eqref{e:SubSeqCon}. To this end, it suffices to prove the following asymptotic tightness criterion (\cite[Theorem 2.1]{ko06}):
\begin{itemize}
\item[(i)] For any $0 \le t_1<t_2<...<t_n \le 1$ with $n \in \N$, the distribution of $(\xi_\e(t_1),...,\xi_\e(t_n))_{0 \le \e \le 1}$ is tight.
\item[(ii)] For all $\lambda>0$
\Be
\lim_{\delta \rightarrow 0} \lim_{\e \rightarrow 0} \sup \PP \left\{\sup_{\stackrel{0 \le s<t \le 1}{|t-s| \le \delta}} |\xi_\e(t)-\xi_\e(s)|>\lambda\right\}=0.
\Ee
\end{itemize}

First of all, for all $\nu>0$, by H\"{o}lder inequality and Lemma \ref{l:LDPUe} we have
\Bes
\begin{split}
\left[\sup_{0 \le t \le 1}\overline{u_{\e}^{\alpha}}(t)\right]^{\nu} & \le  \left(\sup_{0 \le t \le 1}\|u_{\e}(t)\|^\ell_{L^\ell}\right)^{\frac{\nu \alpha}{\ell}} \le C \left[\re^{c_1 B^{*}_1}+\re^{c_2 B^{*}_1} \left(\mcl M^*_{\e,\delta}\right)^{c_3}\right],
\end{split}
\Ees
where $c_1,c_2, c_3$ all depend on $\alpha, \beta, p, q, \nu$. Thanks to Lemma \ref{e:M*BDG} and \eqref{e:DenB*t}, using H\"{o}lder and the above inequalities we have
\Be \label{e:UAlpEst}
\E \left[\sup_{0 \le t \le 1}\overline{u_{\e}^{\alpha}}(t)\right]^{\nu} \le C.
\Ee
Thanks to \eqref{e:XiehGeInf} and \eqref{e:DenB*t}, by similar but easier argument we get
\Be \label{e:XiNuEst}
\E \left[\inf_{0 \le t \le 1} \xi_{\e}(t)\right]^{-\nu} \le C.
\Ee
By H\"{o}lder inequality and \eqref{e:XiehLeSup} we have
\Bes
\begin{split}
\sup_{0 \le t \le 1}\xi_{\e}^{2}(t) &\lesssim \re^{2N+2 \sqrt{\e} B^{*}_1}+\re^{2N+4\sqrt{\e} B^{*}_1}\left(\sup_{0 \le t \le 1}\overline{u_{\e}^{\alpha}}(t)\right)^{\frac{2}{1+\beta}} \\
&\lesssim \re^{2N+2 B^{*}_1}+\re^{2N+4 B^{*}_1}\left(\sup_{0 \le t \le 1}\overline{u_{\e}^{\alpha}}(t)\right)^{\frac{2}{1+\beta}}.
\end{split}
\Ees
Thanks to \eqref{e:UAlpEst} and \eqref{e:DenB*t}, using H\"{o}lder and the above inequalities we have
\Be \label{e:SupXiEEst}
\begin{split}
\E \sup_{0 \le t \le 1}\xi_{\e}^{2}(t) \le C.
\end{split}
\Ee
For all small $c>0$, choosing $K=\sqrt{\frac{C}{c}}$,  by the Chebyshev inequality there exists some $K>0$ such that
\Bes
\PP\left(\sup_{0 \le t \le 1}\xi_{\e}(t)>K\right) \le \frac{\E\sup_{0 \le t \le 1}\xi_{\e}^{2}(t)}{K^2}=c
\Ees
and thus
\Bes
\PP\left(\sup_{0 \le t \le 1}\xi_{\e}(t) \le K\right) \ge 1-c.
\Ees
For any $0 \le t_1<t_2<...<t_n \le 1$ with $n \in \N$, we have
$$\PP\left(\xi_{\e}(t_1) \le K,...,\xi_{\e}(t_n) \le K\right) \ge 1-c.$$
Since $c>0$ is arbitrary, the distribution of $(\xi_{\e}(t_1),...,\xi_{\e}(t_n))$ is tight. Hence, (i) above holds.

Next we check that (ii) also holds. Observe
\Be  \label{e:Tig2}
\begin{split}
\sup_{|s-t| \le \delta} |\xi_\e(t)-\xi_\e(s)| & \le \delta \left[\sup_{0 \le t \le 1} \xi_\e(t)+\sup_{0 \le t \le 1} \frac{\overline{u_\e^{\alpha}}(t)}{\xi_\e^{\beta}(t)}\right] \\
&+\sup_{|s-t|\le \delta} \left|\int_s^t \xi_\e(r) \dot h_\e(s)\dif s\right|+2 \sqrt{\e} \sup_{0 \le t \le 1} \left|\int_0^t \xi_\e \dif B_s\right|.
\end{split}
\Ee
By H\"{o}lder inequality, we get
\Bes
\begin{split}
\sup_{|s-t|\le \delta} \left|\int_s^t \xi_\e(r) \dot h_\e(s)\dif s\right| & \le \sup_{|s-t|\le \delta} \left[\int_s^t \xi^2_\e(r) \dif s\right]^{\frac 12} \|h_\e\|_H \le N \sqrt{\delta} \sup_{0 \le t \le 1} \xi_\e(t),
\end{split}
\Ees
which, together with \eqref{e:SupXiEEst}, yields
\Be \label{e:XiHeEst}
\begin{split}
\E \sup_{|s-t|\le \delta} \left|\int_s^t \xi_\e(r) \dot h_\e(s)\dif s\right|  \le C \delta^{\frac 12}.
\end{split}
\Ee
Observe
\Bes
\begin{split}
\sup_{0 \le t \le 1} \frac{\overline{u_\e^{\alpha}}(t)}{\xi_\e^{\beta}(t)} \le \left(\sup_{0 \le t \le 1} \overline{u_\e^{\alpha}}(t)\right)
\left(\inf_{0 \le t \le 1} {\xi_\e^{-\beta}(t)}\right),
\end{split}
\Ees
by \eqref{e:UAlpEst}, \eqref{e:XiNuEst} and H\"{o}lder inequality, this further gives
\Be \label{e:ExpU/Xi}
\begin{split}
\E \sup_{0 \le t \le 1} \frac{\overline{u_\e^{\alpha}}(t)}{\xi_\e^{\beta}(t)} \le C.
\end{split}
\Ee
Moreover, by H\"{o}lder and martingale inequalities and It$\hat{o}$ identity we get
\Be  \label{e:MarTerm}
\begin{split}
\E \sqrt{\e} \sup_{0 \le t \le 1} \left|\int_0^t \xi_\e \dif B_s\right| & \le  \sqrt{\e}\left[\E \sup_{0 \le t \le 1} \left|\int_0^t \xi_\e \dif B_s\right|^2 \right]^{\frac 12} \\
& \le \sqrt{2\e}\left[\E\left|\int_0^1 \xi_\e \dif B_s\right|^2 \right]^{\frac 12}=\sqrt{2\e}\left[\int_0^1 \E|\xi_\e|^2 \dif s \right]^{\frac 12} \le C \sqrt{\e}
\end{split}
\Ee
where the last inequality is  by \eqref{e:SupXiEEst}. Combining \eqref{e:SupXiEEst}, \eqref{e:ExpU/Xi}, \eqref{e:XiHeEst}, \eqref{e:MarTerm} with \eqref{e:Tig2}, we immediately obtain
\Bes
\begin{split}
\E \sup_{|s-t| \le \delta} |\xi_\e(t)-\xi_\e(s)| \le C(\delta+\sqrt{\delta}+\sqrt{\e}).
\end{split}
\Ees
By Chebyshev inequality,
\Bes
\PP \left\{\sup_{\stackrel{0 \le s<t \le 1}{|t-s| \le \delta}} |\xi_\e(t)-\xi_\e(s)|>\lambda\right\} \le C\lambda^{-1}(\delta+\sqrt{\delta}+\sqrt{\e}),
\Ees
which immediately implies (ii).

\end{proof}

\subsection{Proof of the large deviation theorem}
\begin{proof}
By Theorem 4.4 in \cite{BD}, and Proposition \ref{p:LDP2} and Proposition \ref{p:LDP3}, we can obtain Theorem \ref{T:LDP}.
The $I$ in the theorem is an immediate consequence of \cite[(4.3)]{BD}.
\end{proof}

\section{Discussion and Outlook}

%We have established results on the dynamics for the stochastic shadow Gierer-Meinhardt system. In particular, we have proved the global existence of solutions and
%prove a large deviation result.

%ADD MORE DISCUSSION OF OUR RESULTS. COMPARE OUR RESULTS WITH OTHER RESULTS IN THE LITERATURE.

Finally, let us mention some directions of our future research on the stochastic Gierer-Meinhardt system. Some important questions have been left open in this study and we plan to explore them next. When does blow-up of solutions occur? Can related results be derived for stochastic processes other than one-dimensional standard Brownian motion? Can our results be extended from the stochastic shadow Gierer-Meinhardt system to the full Gierer-Meinhardt system? Do similar results hold for other pattern-forming systems such as the Gray-Scott or Schnakenberg models?

 For pattern formation in the deterministic Gierer-Meinhardt model many interesting phenomena have been established such as Turing instability, peaked steady states with single or multiple spikes, and various kinds of bifurcations. We are interested in the question what will happen if some random forces are added to these models.  Due to the randomness in the system, the peaked patterns and their bifurcations will be random rather than deterministic and we expect that the nature of their interactions will change. Depending on the exact conditions, they can be destabilised by the stochastic effects and new patterns can emerge. Our next goal is to investigate the trajectories of random patterns and their bifurcations and gain further insight into the mechanisms controlling these interactions (\cite{wx14}).

\end{document}